\documentclass{amsart}
 \usepackage{graphicx}

\makeatletter
 \def\LaTeX{\leavevmode L\raise.42ex
   \hbox{\kern-.3em\size{\sf@size}{0pt}\selectfont A}\kern-.15em\TeX}
\makeatother

\newcommand{\BibTeX}{{\rm B\kern-.05em{\sc
i\kern-.025emb}\kern-.08em\TeX}}
\newtheorem{col}{Corollary}[section]
\newtheorem{thm}{Theorem}[section]
\newtheorem{lem}[thm]{Lemma}

\newtheorem{rem}[thm]{Remark}
\theoremstyle{defn}
\newtheorem{defn}{Definition}

\numberwithin{equation}{section}

\newcommand{\Ltwo}{\ell^2}

\begin{document}

\title[ Exact positive cubature formulas on graphs ]{Exact positive cubature formulas via generalized sampling  on combinatorial  graphs}

\author{Isaac Z. Pesenson }

\address{Department of Mathematics, Temple University,
 Philadelphia,
PA,  19122 USA}
\email{pesenson@temple.edu}

\maketitle

\begin{abstract}

We consider a disjoint cover (partition) of an undirected weighted finite graph $G$ by $|J|$ connected subgraphs (clusters)  $\{S_{j}\}_{j\in J}$ and select  a function $\xi_{j}$ on each of the clusters. For a given signal $f$ on $G$   the set of its weighted average values samples is defined via inner products $\{\langle  \xi_{j}, f\rangle\}_{j\in J}$. The goal of the paper is to establish exact quadrature formulas with positive weights  which are exploring  these samples generated by  bandlimited functions.

\end{abstract}

\section{Introduction }

\label{sect:main}

During last years in connection with a variety of important applications, the problems about sampling, interpolation and cubature formulas in the setting of combinatorial graphs  attracted attention of many mathematicians and engineers. Here are some of the relevant papers \cite{Bab}-\cite{BabT}, \cite{Cav},   \cite{CM}, \cite{Erb}, \cite{HL}, \cite{LS}, \cite{Ortega}-\cite{WCG}.

In the present paper we consider a disjoint cover (partition) of an undirected weighted finite graph $G$ by a family of connected subgraphs (clusters)  $\mathcal{S}=\{S_{j}\}_{j\in J}$ and select  a set of functions $\Xi=  \{\zeta_{j}\}_{j\in J}$ on each of the clusters. We assume that each of the functions $\{\zeta_{j}\}_{j\in J}$  is not negative. For a given signal $f$ on $G$   the set of its weighted average values samples is defined via inner products $\{\langle  \zeta_{j}, f\rangle\}_{j\in J}$. 
For example, when $\zeta_{j}$ is a Dirac measure $\delta_{v_{j}}$ supported at a vertex $v_{j}\in S_{j}$, then 
$$
\langle \zeta_{j}, f\rangle= f(v_{j}),
$$
is just a value of $f$ at $v_{j}$. When $\zeta_{j}$ is a characteristic function of a subset $U_{j}\subseteq S_{j}$ then
$$
\langle \zeta_{j}, f\rangle=\frac{1}{|U_{j}|}\sum_{v\in U_{j}}f(v),
$$
is the average value of $f$ over $U_{j}$. 

Let's consider subspaces of bandlimited functions of bandwidth $\omega>0$ which are denoted as ${\bf E}_{\omega}(G)$  (see below for all the definitions). These subspaces are defined in terms of the eigenvalues of  a combinatorial Laplace operator $L$ on $G$.  The goal of the paper is to establish exact cubature formulas with positive weights $\{\varpi_{j}\}_{j\in J},$
  \begin{equation}\label{01}
\sum _{v\in V(G)}f(v)= \sum_{j\in J} \varpi
_{j}\langle \zeta_{j}, f\rangle,\>\>\>\>f\in {\bf E}_{\omega}(G),
 \end{equation}
 which are exploring  samples $ \{\langle  \zeta_{j}, f\rangle\}_{j\in J},$ generated by  functions from subspaces of bandlimited functions of bandwidth $\omega>0$.  We prove that there exists a such cut off frequency $\omega_{0}=\omega_{0}\left(G, \mathcal{S}, \Xi\right) $ that the formula (\ref{01}) holds for all functions in subspaces  ${\bf E}_{\omega}(G)$ with $0\leq \omega<\omega_{0}$.  However, in general there is not guarantee that for a particular graph $G$ and a choice of a pair   $\left(\mathcal{S}, \Xi\right)$ the interval $[0, \omega_{0}]$ contains non-trivial eigenvalues of a relevant Laplace operator $L$. In this connection we argue in Lemma \ref{interval}  that there are "many" \textit{community graphs} (see \cite{F}) for which  the interval $[0,\>\omega_{0}]$ contains "many" nontrivial frequencies of the corresponding  operator $L$. Moreover, we claim that it is a rather typical situation  for community graphs for which connections between communities are sufficiently weak compared to connections inside of  communities. 
 
 In section \ref{Analysis} we introduce several basic notions we use in the paper. Section \ref{finite} devoted to Poincare-type inequalities on finite graphs. In section \ref{Poinc-on-large} we extend Poincare-type inequalities to "large" finite graphs partitioned into "small" subgraphs $\{S_{j}\}$ (see Corollary \ref{PP-2}). 
This  Corollary implies that projections of  the relevant functionals onto an appropriate space ${\bf E}_{\omega}(G)$ form a Hilbert frame in this space. This fact  implies, in turn, that every $f\in  {\bf E}_{\omega}(G)$ can be stably reconstructed from a set of relevant samples. There exist several ways for a such reconstruction. Each of them can be used to obtain some exact and approximate methods for estimating   averages like $\sum_{v\in V(G)} f(v)$ (see  \cite{PPF, PesM21}). In section \ref{norm-1} we obtain Poincare-type inequalities in $\ell^{1}(G)$ norm. Our cubature formulas are proved in section \ref{cub-fla}.  It is worth to note that only in the proof of Lemma \ref{lowerbound}         we are using the assumption that the functions $\{\zeta_{j}\}_{j\in J}$ (which are used to define functionals $\{\langle \zeta_{j}, \cdot\rangle\}_{j\in J}$) are not negative. Section \ref{Appendix} contains a proof of Lemma \ref{Lemma} and the definition of Hilbert frames. 
Methodologically  this paper is a combination of  several ideas from the articles \cite{CKP} and \cite{PesM21}.

\section{Analysis  of graph signals}\label{Analysis}

Let $G$ denote an undirected weighted graph, with  a finite  number of vertices $V(G)$  and weight function $w: V(G) \times V(G) \to [0, \infty)$, $\>\>w$ is symmetric, i.e., $w(u,v) = w(v,u)$, and $w(u,u)=0$ for all $u,v \in V(G)$. The edges of the graph are the pairs $(u,v)$ with $w(u,v) \not= 0$. 
Let $\ell^{2}(G)\>\>$ denote the space of  all complex-valued functions with the inner product
$$
\left<f,g\right>=\sum_{v\in V(G)}f(v)\overline{g(v)}
$$
and  the norm
\[
\| f \|_{\ell^{2}(G)} = \left( \sum_{v \in V(G)} |f(v)|^2 \right)^{1/2}.
\] 
We will also need the space $\ell^{1}(G)$ with the norm
\[
\| f \|_{\ell^{1}(G)} =  \sum_{v \in V(G)} |f(v)|.
\] 
For a set $S\subset V(G)$ the notation $\ell^{2}(S)\left(=\ell^{1}(S)\right)$ is used for all functions in $\ell^{2}(G)$ supported on $S$. We intend to prove certain  Poincar\'e-type estimates involving weighted gradient norm.
\begin{defn}The  weighted gradient norm of a function $f$ on $V(G)$ is defined by 
\begin{equation}\label{gr}
 \| \nabla f \|_{\ell^{2}(S)} = \left( \sum_{u, v \in V(G)} \frac{1}{2} |f(u) - f(v)|^2 w(u,v) \right)^{1/2}.
\end{equation} 
\end{defn}

In  the case of a finite graph and   the weighted Laplace operator $L: L_{2}(G) \mapsto L_{2}(G)$  is introduced  via
\begin{equation}\label{L}
 (L f)(v) = \sum_{u \in V(G)} (f(v)-f(u)) w(v,u)~.
\end{equation}
We will also need the following equality which holds true for all  finite graphs 
\begin{equation}\label{L-G}
\|L^{1/2}f\|_{\ell^{2}(S)}=\|\nabla f  \|_{\ell^{2}(S)},\>\>\>\>\>f\in\ell^{2}(G).
\end{equation}

We are using the spectral theorem for the operator $L$ to introduce the associated Paley-Wiener spaces which are also known as the spaces of bandlimited functions.
\begin{defn}\label{PW}
 The Paley-Wieners space ${\bf E}_{\omega}(G) \subset \Ltwo(G)$ is the image space of the projection operator ${\bf 1}_{[0,\>\omega]}(L)$ (to be understood in the sense of Borel functional calculus). For a given $f\in {\bf E}_{\omega}(G)$ the smallest $\omega\geq 0$ such that $f\in {\bf E}_{\omega}(G)$ is called the bandwidth of $f$.
 \end{defn}
By using the Spectral theorem one can show   that a function $f$ belongs to the space  ${\bf E}_{\omega}(G)$ if and only if for every positive $t>0$ the following Bernstein inequality holds
\begin{equation}\label{B}
\|L^{t}f\|_{\ell^{2}(S)}\leq \omega^{t}\|f\|_{\ell^{2}(S)},\>\>\>t>0.
\end{equation}
For  a finite  graph $G$ its Laplacian  $L$ has a discrete spectrum   $0=\lambda_{0}<\lambda_{1}\leq...\leq \lambda_{|G|-1}$, and multiplicity of $0$ is the number of connected components of $G$. A set of corresponding orthonormal eigenfunctions will be denoted as $\{\varphi_{j}\}_{j=0}^{|G|-1}$. In this case the space ${\bf E}_{\omega}(G)$ coincides with
$$
span \{\varphi_{0}, ..., \varphi_{k} : \lambda_{k}\leq \omega,\>\>\lambda_{k+1}>\omega\}.
$$
We will also need the following definition. 
\begin{defn}\label{3}
For a given graph $G$ and a given $\tau\geq 0$  let $ \mathcal{X}_{\tau}(G) \subset \ell^{2} (G)$ denote the subset of all $f\in \ell^{2}(G)$ fulfilling the inequality $\| \nabla f \|_{\ell^{2}(S)} \le \sqrt{\tau} \| f \|_{\ell^{2}(S)}$ . 
\end{defn}
Although the sets $\mathcal{X}_{\tau}(G)$ are  closed with respect to multiplication by constants  they are not linear spaces (in contrast to ${\bf E}_{\omega}(G)$). Clearly,  for every $\omega\geq 0$ one has the inclusion  ${\bf E}_{\omega}(G)\subset \mathcal{X}_{\omega}(G)$. However, the space ${\bf E}_{\omega}(G)$ can be trivial but   the sets $\mathcal{X}_{\tau}(G)$ are never trivial (think about a  function whose norm is  bigger than its  variation). 
 Note also that  if $f$ has a very large norm its variations $\|\nabla f\|_{\ell^{2}(S)}$ can be large  too even if $\tau$ is small.

\section{A Poincare-type inequality for finite graphs. }\label{finite}

For a finite connected graph $G$ which contains more than one vertex let   $\Psi$ be a functional on $\ell^{2}(G)$ which is defined by a function $\psi\in \ell^{2}(G)$, i.e. $$
 \Psi(f)=\langle\psi, f\rangle=\sum_{v\in V(G)}\psi(v)\overline{f(v)}.
 $$
Note, that the normalized eigenfunction $\varphi_{0}$ which corresponds to the eigenvalue $\lambda_{0}=0$ is given by the formula 
$$
\frac{\chi_{G}}{\sqrt{|G|}}=\varphi_{0},
$$
 where $\chi_{G}(v)=1$ for all $v\in V(G)$. 

\begin{thm}\label{fund-Th-0}
Let $G$ be a finite connected graph which contains more than one vertex and $\Psi(\varphi_{0})=\langle \psi, \varphi_{0}\rangle$ is not zero.  If $f\in Ker(\Psi)$ then
\begin{equation}\label{fund-ineq}
\|f\|^{2}_{\ell^{2}(G)}\leq \frac{\theta}{\lambda_{1}}\|\nabla f\|^{2}_{\ell^{2}(G)},\>\>\>f\in Ker(\Psi),
\end{equation}
where $\lambda_{1}$ is the first non zero eigenvalue of the Laplacian (\ref{L}) and
\begin{equation}\label{Teta}
 \theta=    \frac{\|\psi\|_{\ell^{2}(G)}^{2}}{\left| \left< \psi, \varphi_{0}\right>\right|^{2}}.
 \end{equation}

\end{thm}

\begin{col}\label{col}
If $\Psi(\varphi_{0})=\langle \psi, \varphi_{0}\rangle$ is not zero then the following inequality holds for every $f\in \ell^{2}(G)$ 
\begin{equation}\label{graph-Poinc-0}
\left\|f-\frac{\Psi(f)} {\Psi(\varphi_{0})}\varphi_{0}\right\|_{\ell^{2}(G)}^{2}\leq \frac{\|\psi\|_{\ell^{2}(G)}^{2}}{\lambda_{1}\left| \Psi(\varphi_{0})\right|^{2}}\|\nabla f\|_{\ell^{2}(G)}^{2}.
\end{equation}
\end{col}

 Theorem  \ref{fund-Th-0} is  a particular case of the following general fact.

\begin{lem}\label{Lemma}

Let $T$ be a non-negative self-adjoint bounded operator  with a discrete spectrum (counted with multiplicities) $0=\lambda_{0}<\lambda_{1}\leq .... $ in a Hilbert space $H$. Let  $\varphi_{0}, \varphi_{1},...,$ be a corresponding set of orthonormal eigenfunctions which is a basis in $H$. 
For any non-trivial $\psi\in H$ let $H_{\psi}^{\bot}$ be a subspace of all $f\in H$ which are orthogonal to $\psi$.
If $f\in H_{\psi}^{\bot}$ then

\begin{equation}\label{fund-ineq-10}
 \lambda_{1}\frac{\left| \left< \psi, \varphi_{0}\right>\right|^{2}}{\|\psi\|_{H}^{2}}\|f\|^{2}_{H}\leq \|T f\|_{H}^{2}.
\end{equation}
\end{lem}

Applying this Lemma to the operator $L^{1/2}$ with eigenvalues $\lambda_{k}^{1/2}$ and using equality (\ref{L-G}) we obtain Theorem \ref{fund-Th-0}.

\begin{thm}\label{Fund-1}
If $\Psi(\varphi_{0})=\langle \psi, \varphi_{0}\rangle$ is not zero then the following inequality holds for every $f\in \ell^{2}(G)$ and every $\epsilon>0$
 \begin{equation}\label{one-set}
\|f\|^{2}_{\ell^{2}(G)}\leq (1+\epsilon) \frac{\|\psi\|_{\ell^{2}(S)}^{2}}{\lambda_{1}\left| \Psi(\varphi_{0})\right|^{2}}\|\nabla f\|^{2}_{\ell^{2}(G)}  +
  \frac{1+\epsilon}{\epsilon}\frac{1}{|\Psi(\varphi_{0})|^{2}}|\langle \psi,  f\rangle|^{2}.
 \end{equation}
\end{thm}

\begin{proof}

We will need the inequality 
\begin{equation}\label{algebra}
|X|^{2}\leq        (1+\epsilon)           \left|X-Y\right|^{2}+  \frac{1+\epsilon}{\epsilon}\left|Y\right|^{2},\>\>\epsilon>0, 
\end{equation}
which follows from the inequalities
$
|X|^{2}\leq |A-B|^{2}+2|A-B||B|+|B|^{2},
$
and
$
2|A-B||B|\leq\epsilon^{-1}|X-Y|^{2}+\alpha|Y|^{2},\>\>\>\>\epsilon>0.
$
Using  inequality
(\ref{algebra}) we obtain
$$
\|f\|_{\ell^{2}(G)}^{2}\leq 
  (1+\epsilon)     \left\|f-\frac{\langle \psi,  f\rangle}{\Psi(\varphi_{0})}\varphi_{0}\right \|_{\ell^{2}(G)}^{2}+ \frac{1+\epsilon}{\epsilon}    \frac{\left|\langle \psi,  f\rangle\right|^{2}}{\left|\Psi(\varphi_{0})\right|^{2}}.
$$
Note, that if $\Psi(f)= \left<\psi, \>\varphi_{0}\right>\neq 0$ then $f-\frac{\left<\psi,\>f\right>}{\left<\psi, \>\varphi_{0}\right>}\varphi_{0}$ belongs to $H_{\psi}^{\bot}$. According to   Corollary \ref{col}   the first term on the right is dominated by 
\begin{equation}\label{graph-Poinc-0}
(1+\epsilon) \frac{\|\psi\|_{\ell^{2}(S)}^{2}}{\lambda_{1}\left| \Psi(\varphi_{0})\right|^{2}}\|\nabla f\|_{\ell^{2}(G)}^{2}.
\end{equation}
The proof is complete.

\end{proof}

\section{Generalized  Poincare-type inequalities for finite and infinite graphs and sampling theorems.}\label{Poinc-on-large}
\subsection{Generalized  Poincare-type inequalities in $\ell^{2}(G)$.}

For a finite  $G$  we consider the following assumptions and notations.

\bigskip

\textit{We assume that 
$\mathcal{S}=\{S_{j}\}_{j\in J}$ form a disjoint cover of $V(G)$
\begin{equation}\label{cover}
 \bigcup_{j\in J}S_{j}=V(G).
\end{equation}}

\textit{Let $L_{j}$  be the Laplacian  for the {\bf induced subgraph} $S_{j}$. In order to insure that $L_{j}$ has at least one non zero eigenvalue, we assume  that every $S_{j}\subset V(G),\>\>j\in J, $ is a \textit{finite and connected subset of vertices with more than one vertex. } The spectrum of the operator  $L_{j}$ will be denoted as $0=\lambda_{0,j}< \lambda_{1, j}\leq ... \leq \lambda_{|S_{j}|, j} $ and the corresponding o.n.b. of eigenfunctions as $\{\varphi_{k,j}\}_{k=0}^{|S_{j}|}$. Thus the first non-zero eigenvalue for a subgraph $S_{j}$ is $\lambda_{1, j}$.}

\textit{Let $\|\nabla_{j}f_{j}\|_{\ell^{2}(S_{j})}$ be the weighted gradient for the  induced subgraph  $S_{j}$.
With every $S_{j},\>\>j\in J,$ we associate a function $\psi_{j}\in \ell^{2}(G)$ whose support is in $S_{j}$
and introduce  the  functionals $\Psi_{j}$ on $\ell^{2}(G)$ defined by these functions
\begin{equation}\label{functional}
\Psi_{j}(f)= \langle \psi_{j}, f\rangle=\sum_{v\in V(S_{j})}  \psi_{j}(v)     \overline{f(v)},\>\>\>\>f\in \ell^{2}(G).
\end{equation}
Notation $\chi_{j}$ will be used for the characteristic function of   $S_{j}$ and  we use $f_{j}$ for $f\chi_{j},\>\>f\in \ell^{2}(G)$. }

\bigskip

As usual, the induced graph $S_{j}$ has the same vertices as the set $S_{j}$ but only such edges of $E(G)$ which have both ends in $S_{j}$.
 We define the following notations

$$
\theta_{j}=     \frac{\langle \psi_{j}, f\rangle} {\langle \psi_{j}, \varphi_{0,j}\rangle},\>\>\>\>\>\>\>
\Xi=\left(\{S_{j}\}_{j\in J},\>\{\psi_{j}\}_{j\in J}\right),
$$
and
$$
\Theta_{j}
=\frac{ \|\psi_{j}\|^{2} }  { \lambda_{1,j}\left| \left< \psi_{j}, \varphi_{0, j}\right>\right|^{2}},\>\>\>\>\>\>\Theta_{\Xi}=\max_{j\in J}\Theta_{ j},
$$ 
and then introduce the function
$$
\zeta_{j}=\frac{\psi_{j}}{|S_{j}|^{1/2}\langle \psi_{j}, \varphi_{0,j}\rangle},
$$
which defines the functional
\begin{equation}\label{zetas}
     \langle \zeta_{j}, f\rangle =  \frac{\langle \psi_{j}, f\rangle }{|S_{j}|^{1/2}\langle \psi_{j}, \varphi_{0,j}\rangle}.
\end{equation}
Applying  Corollary  \ref{col} and Theorem \ref{Fund-1} to every $\nabla_{j}$ in the space $\ell^{2}(S_{j})$  and using the relations
\begin{equation}\label{loc-glob}
\sum_{j\in J}\|\nabla_{j}f_{j}\|_{\ell^{2}(G)}^{2}\leq \|\nabla f\|_{\ell^{2}(G)}^{2}=\|L^{1/2}f\|_{\ell^{2}(G)}^{2},
\end{equation}
we obtain the following result.
\begin{thm}\label{Main-Th}
Let $G$ be a connected finite   graph and $\mathcal{S}=\{S_{j}\}$ is its disjoint cover by finite sets.  Let $L_{j}$  be the Laplace operator of the  induced  subgraph $S_{j}$ whose first nonzero eigenvalue is $\lambda_{1, j}$ and $\varphi_{0, j}=1/\sqrt{|S_{j}|}$ is its normalized eigenfunction with eigenvalue zero. Assume that  for every $j$  function $\psi_{j}\in  \ell^{2}(G)$ has support in $S_{j}$,  $\>\>\>\Psi_{j}(f)= \langle \psi_{j}, f \rangle,$ and $\Psi_{j}(\varphi_{0,j})=\langle \psi_{j}, \varphi_{0, j}\rangle \neq 0$. Then the following inequality  holds true 
for every $f\in \ell^{2}(G)$ and $f_{j}=f\chi_{j}$

\begin{equation}\label{graph-Poinc-00}
\sum_{j\in J} \sum_{v\in V(S_{j})} \left|f_{j}(v)-     \langle \zeta_{j}, f\rangle\chi_{j}(v) \right|^2              \leq               \Theta_{\Xi}   \|\nabla f\|_{\ell^{2}(G)}^{2},\>\>f\in \ell^{2}(G),
\end{equation}
\end{thm}

Set
$
A_{\Xi}=\max_{j\in J}\left(\frac{\theta_{j}}{\lambda_{1,j}}\right).
$
Using (\ref{graph-Poinc-00}) the next   statement about a  Plancherel-Polya-type  (or Marcinkiewicz-Zygmund-type, or frame) inequalities can be proved (see \cite{PesM21}).

\begin{col}\label{PP-2} Assume that all the assumptions of Theorem \ref{Main-Th} hold 
and
$$
a=\max_{j\in J} \frac{1}{|\Psi_{j}(\varphi_{0, j})|^{2}}, \>\>c=\max_{j\in J}\|\psi_{j}\|^{2}.\>\>
$$
For every $f\in {\bf E}_{\omega}(G)$ with $\omega$ satisfying 
\begin{equation}\label{main-cond}
0\leq \omega<\frac{1}{A_{ \Xi   }      },\>\>\>\>\>\>
\end{equation}
the Plancherel-Polya-type inequalities hold true
\begin{equation}\label{P-P-11}
\frac{(1-\mu)\tau}{(1+\tau)a}             \|f\|^{2}\leq  \sum_{j\in J}|\langle \zeta_{j}, f\rangle|^{2}\leq c\|f\|^{2},
\end{equation}
for those $\tau>0$ for which the inequality 
\begin{equation}\label{2main-cond}
\mu=\omega(1+\tau)A_{  \Xi   } <1,
\end{equation}
holds.
\end{col}

\begin{rem}
This  Corollary implies that projections of functionals onto the space ${\bf E}_{\omega}(G)$ where $ \omega$ satisfies (\ref{2main-cond}) form a Hilbert frame in this space. This fact  implies, in turn, that every $f\in  {\bf E}_{\omega}(G)$ can be reconstructed from the values $\{\langle \zeta_{j}, f\rangle\}$. There exist several ways for a such reconstruction. Each of them can be used to obtain some approximate methods for estimating   averages  $\sum_{v\in V(G)} f(v)$ (see  \cite{PPF, PesM21}).
\end{rem}

\section{Inequalities in the $\ell^{1}(G)$ norm}\label{norm-1}

\begin{thm}\label{norm-one}
For a given $\omega>0$ the following inequality holds 
 for every $f\in \mathcal{X}_{\omega}(G)$ 
\begin{equation}\label{fla-norm-one}
  \sum_{j\in J}\sum_{v\in V(S_{j})} \left|f_{j}(v)-     \langle \zeta_{j}, f\rangle\chi_{j}(v)\right| \leq \sqrt{\omega}\>  C_{\Xi}  \sum_{v\in G}|f(v)|  ,\>\>\>\>\>f_{j}=f|_{S_{j}},
 \end{equation}
where  
\begin{equation}\label{const}
C_{\Xi}=\sqrt{    |J|\max_{j\in J}\left(  |S_{j}|\Theta_{ j} \right)  }.
\end{equation}
\end{thm}

\begin{proof}
According to Corollary \ref{col} and Theorem \ref{Main-Th} we have 
$$
 \sum_{v\in V(S_{j})} \left|f_{j}(v)-     \langle \zeta_{j}, f\rangle\chi_{j}(v)\right|^2              \leq   \Theta_{j}\|\nabla_{j} f_{j}\|_{\ell^{2}(S_{j})}^{2},
$$
and since 
$$
 \sum_{v\in V(S_{j})} \left|f_{j}(v)-    \langle \zeta_{j}, f\rangle\chi_{j}(v)\right| \leq 
 $$
 $$
|S_{j}|^{1/2}\left(\sum_{v\in V(S_{j})} \left|f_{j}(v)-      \langle \zeta_{j}, f\rangle\chi_{j}(v)\right|^2             \right)^{1/2}\leq \sqrt{|S_{j}|\Theta_{ j}} \> \|\nabla_{j} f_{j}\|_{\ell^{2}(S_{j})}, 
$$
we obtain
$$
\sum_{j\in J} \sum_{v\in V(S_{j})} \left|f_{j}(v)-     \langle \zeta_{j}, f\rangle\chi_{j}(v)\right| \leq \sum_{j\in J} \sqrt{|S_{j}|\Theta_{ j}}  \|\nabla_{j} f_{j}\|_{\ell^{2}(S_{j})}\leq
$$
$$
\sqrt{|J|\max_{j\in J}\left(|S_{j}|\Theta_{ j}\right)}  \|\nabla f\|_{\ell^{2}(G)}, 
$$
where we used the inequality 
$$
   \sum_{j\in J}   \|\nabla_{j} f_{j}\|_{\ell^2(S_{j})}   \leq \sqrt{|J|}\|\nabla f\|_{\ell^2(G)}.
$$
After all, due to (\ref{L-G}), (\ref{B}) and the inequality $\|f\|_{\ell^{2}(G)}\leq \|f\|_{\ell^{1}(G)}$
 we have
$$
  \sum_{j\in J}\sum_{v\in V(S_{j})} \left|f_{j}(v)-    \langle \zeta_{j}, f\rangle\chi_{j}(v)\right| \leq  \sqrt{\omega}\>C_{\Xi}  \sum_{v\in G}|f(v)|,  \>\>\>f\in \mathcal{X}_{\omega}(G),\>\>\>f_{j}=f|_{S_{j}},
 $$
here 
$
C_{\Xi}=\sqrt{    |J|\max_{j\in J}\left(  |S_{j}|\Theta_{ j} \right)  }.
$
Theorem is proven.

\end{proof}

Since ${\bf E}_{\omega}(G)\subset \mathcal{X}_{\omega}(G)$ we have the following corollary from the last theorem.
\begin{col}\label{corollary-gamma}
If
$
\gamma=\sqrt{\omega}\> C_{\Xi} 
$
then for all $f\in {\bf E}_{\omega}(G)$ 
\begin{equation}\label{fla-5.2}
  \sum_{j\in J}\sum_{v\in V(S_{j})} \left|f_{j}(v)-   \langle \zeta_{j}, f\rangle\chi_{j}(v)\right| \leq \gamma \sum_{v\in G}|f(v)|.  
\end{equation}

\end{col}

Consider a finite dimensional  space $\mathbb{R}^{|J|}$ of all sequences $\{\alpha_{j}\},\>1\leq j\leq |J|,$ equipped with the norm  
\begin{equation}\label{space E}
||| \{\alpha_{j}\}_{j\in J}|||=\sum_{j\in J}|\alpha_{j}|\>|S_{j}|, 
\end{equation}
and define the following \textit{sampling operator}  $Q$ as
\begin{equation}\label{Q}
Q: f\in {\bf E}_{\omega}(G)\mapsto \{f_{j}\}_{j\in J}=\{\langle \zeta_{j}, f\rangle\}_{j\in J}\in \mathbb{R}^{|J|}.
\end{equation}
The Plancherel-Polya  inequality (\ref{P-P-11}) already shows that $Q$ is continues and injective.  However, we will need a more accurate version of (\ref{P-P-11}) in the norm $\ell^{1}(G)$.

\begin{lem} If $\sqrt{\omega}\>C_{\Xi}=\gamma\in (0,\>1/2)$  then 
 the following double inequality holds
\begin{equation}\label{PP-norm-1}
(1-\gamma)\sum_{v\in V(G)}|f(v)|\leq \sum_{j}\langle  \zeta_{j}, f\rangle ||S_{j}|\leq (1+\gamma)\sum_{v\in V(G)}|f(v)|,
\end{equation}
for every $f\in {\bf E}_{\omega}(G)$.
\end{lem}\>
\begin{proof}

According to Corollary   \ref{corollary-gamma}  we obtain for every $f\in {\bf E}_{\omega}(G)$
\begin{equation}\label{P-Main}
\left|\sum_{v\in V(G)} f(v)-\sum_{j}\langle  \zeta_{j}, f\rangle |S_{j}|\right|=\left|\sum_{j}\sum_{v\in S_{j}} f(v)-\sum_{j}  \langle  \zeta_{j}, f\rangle\sum_{v\in S_{j}}\chi_{j}(v) \right|\leq
$$
$$
\sum_{j}\sum_{v\in S_{j}} \left |f(v)-\langle  \zeta_{j}, f\rangle\chi_{j}(v)\right| \leq \gamma\sum_{v\in V(G)}|f(v)|,\>\>\>\>0<\gamma<1/2.
\end{equation}
Starting  with the trivial inequality 
$$
|f(v)|\leq |f(v)-\langle  \zeta_{j}, f\rangle|+|\langle  \zeta_{j}, f\rangle|,
$$
 one obtains the following one
\begin{equation}\label{PD_1}
\sum_{v\in V(G)}|f(v)|=\sum_{j}\sum_{v\in S_{j}}|f(v)|\leq 
$$
$$
\sum_{j}\sum_{v\in S_{j}}|f(v)-\langle  \zeta_{j}, f\rangle\chi_{j}(v) |+\sum_{j}|\langle  \zeta_{j}, f\rangle |\>|S_{j}|.
\end{equation}
Similarly, the inequality 
$$
|\langle \zeta_{j}, f\rangle |\leq |f(v)-\langle  \zeta_{j}, f\rangle|+|f(v)|,
$$
implies the next one
$$
\sum_{v\in S_{j}}|\langle \zeta_{j}, f\rangle |\chi_{j}(v)\leq \sum_{v\in S_{j}}|f(v)-\langle \zeta_{j}, f\rangle \chi_{j}(v)|+\sum_{v\in S_{j}}|f(v)|,
$$
and  then
\begin{equation}\label{PD_2}
\sum_{j}|\langle \zeta_{j}, f\rangle |\>|S_{j}|\leq \sum_{j}\sum_{v\in S_{j}}|f(v)-\langle \zeta_{j}, f\rangle \chi_{j}(v)|+\sum_{v\in V(G)}|f(v)|.
\end{equation}
Together these inequalities imply (\ref{PP-norm-1})  for every $f\in {\bf {E}}_{\omega}(G)$.
The statement is proven. 
\end{proof}

\begin{rem}\label{4.3}  We  note that the inequalities (\ref{graph-Poinc-00}) -(\ref{fla-5.2}), (\ref{PP-norm-1})  and the constant (\ref{const})  {\bf independent of the edges  outside of the clusters $S_{j}$}.
In other words, if one rearranges and mutually connects subgraphs $\{S_{j}\}_{j\in J}$ in any other way in order to obtain a new graph, (and will keep the set of functionals $\{\zeta_{j},\cdot\}$) these  inequalities  and the constant $C_{\Xi}$ would remain the same. 
\end{rem}

For  the rest of the paper it is necessary  to keep the assumption  that the constant $\gamma=\sqrt{\omega}\> C_{\Xi}$ is not greater than $1/2$. In this connection the important question arises:  if the interval $(0,\>C_{\Xi}/2)$ contains non-trivial eigenvalues of the Laplacian $L$ of the entire graph $G$. 

\bigskip

The next Lemma gives a partial answer to this question.

\begin{lem}\label{interval}
For every natural $N$ and every positive $\rho>0$ there exist community graphs with $N$ communities and such that the interval $[0,\>\rho)$ contains at least $N$ eigenvalues (counting with multiplicities) of the Laplace operator.
\end{lem}
\begin{proof}
Suppose that $G$ is a community graph with $N$ communities. Let's remove from $G$ all the edges between communities. It will give a new disconnected graph $\widetilde{G}$ with $N$ components. For this graph $\widetilde{G}$ zero will be an eigenvalue of multiplicity $N$. By slightly perturbing the matrix of the Laplacian $L_{\widetilde{G}}$ (and keeping it symmetric and non-negative) we can construct a graph $\overline{G}$ whose Laplacian $L_{\overline{G}}$  will contain at least $N$ eigenvalues in the interval $[0,\>\rho)$. 
Lemma is proven.
\end{proof}
Clearly, more substantial perturbations can reduce the number of  eigenvalues of a perturbed Laplacian  in the interval $[0,\>\rho)$.

\section{A positive cubature formula}\label{cub-fla}

Suppose $\mathcal{E}$ is a linear normed space, $\mathcal{F }\subset  \mathcal{E}$ is a subspace of $\mathcal{E}$, and $\mathcal{C}$ is a convex cone in $\mathcal{E}$, which determines an order on $\mathcal{E}$ .
The following theorem can be found in \cite{B, N} and it will be used in proving our cubature formulas. 

\begin{thm}\label{BN}(Bauer-Namioka). Let $\Theta$ be a linear form on a subspace $ \mathcal{F}$ of $\mathcal{E}$. There exists a continuous positive linear extension of $\Theta$ to $\mathcal{E}$ if and only if there exists a neighborhood $\mathcal{U}$ of $0$ such
that the set $\Theta(\mathcal{F}\cap(\mathcal{U} + \mathcal{C}))$ is bounded from below.

\end{thm}

We consider the sampling operator  $Q$ which is continuous and injective.
It is worth to emphasize that in general one can expect only an inclusion $\mathcal{R}(Q)\subset \mathbb{R}^{|J|}$ and not an equality 
$\mathcal{R}(Q)=\mathbb{R}^{|J|}.$ The continues inverse operator $Q^{-1}$  is defined on $\mathcal{R}(Q)$ :
\begin{equation}\label{Q}
Q^{-1}\left(\{f_{j}\}_{j\in J}\right)=f\in {\bf E}_{\omega}(G),
\end{equation}
and it means that every sequence $\{f_{j}\}_{j\in J}\in \mathcal{R}(Q)$ is mapped to a unique 
 function in ${\bf  E}_{\omega}(G)$ for which $f_{j}=\langle f, \zeta_{j}\rangle, \>1\leq j\leq |J|$.
 We  pick a number 
$
0<\gamma<1/2
$
 and introduce the functional $\Theta$ on $\mathcal{R}(Q)$  by using the formula
\begin{equation}\label{functional}
\Theta\left(\{f_{j}\}\right)=  \sum_{v\in V(G)}  Q^{-1}\left(\{f_{j}\}\right)(v)  -    \frac{1-2\gamma}{1-\gamma}|||\{f_{j}\}|||=
$$
$$
\sum_{v\in V(G)}f(v) -\frac{1-2\gamma}{1-\gamma}\sum_{j}\langle \zeta_{j}, f\rangle|S_{j}|,\>\>\>\>\{f_{j}\}\in \mathcal{R}(Q),
\end{equation}
where $f=Q^{-1}(\{f_{j}\}) $ is a unique function in ${\bf E}_{\omega}(G)$ for which $\langle f, \zeta_{j}\rangle=f_{j}$ for all $j$.

To apply Theorem \ref{BN} we treat $\left(\mathbb{R}^{|J|},\>|||\cdot |||\right)$ as the space $\mathcal{E}$, the $\left(\mathcal{R}(Q),\>||| \cdot |||\right)$ as the subspace $\mathcal{F}$, and also
$$
\mathcal{C}=\left\{  \{\alpha_{j}\}_{j\in J}\in \mathcal{E} :  \alpha_{j}\geq 0,\>1\leq j\leq |J|  \right\}, \>
\mathcal{U}=\left\{  \{\alpha_{j}\}_{j\in J}\in \mathcal{E} :   ||| \{\alpha_{j}\}_{j\in J}|||\leq 1      \right\}.
$$
 
  In the next lemma we need the assumption that  the functionals $\langle \zeta_{j}, \cdot\rangle $ which were defined in (\ref{zetas}) are positive.

  \begin{lem}\label{lowerbound} 
 If the functionals  $\langle \zeta_{j}, \cdot\rangle, \>j\in J, $ are positive and the positive $\gamma$ in (\ref{corollary-gamma}) is less than $1/2$ then the functional $\Theta$ is  bounded from below on the set  $\mathcal{F}\cap(\mathcal{U} + \mathcal{C})$. Namely, the following estimate holds
$$
\Theta (\mathcal{F}\cap(\mathcal{U} + \mathcal{C}))>\frac{2\gamma}{\gamma-1},\>\>\>\>0<\gamma<1/2.
$$
\end{lem}
\begin{proof}

From (\ref{P-Main}) one has 
$$
-\gamma\sum_{v\in V(G)}|f(v)|\leq
 \sum_{v\in V(G)} f(v)-\sum_{j}\langle  \zeta_{j}, f\rangle|S_{j}|,\>\>f\in {\bf E}_{\omega}(G),
$$
and  from (\ref{PP-norm-1})
$$
 \sum_{v\in V(G)}|f(v)|\leq\frac{1}{1-\gamma}\sum_{j}|\langle \zeta_{j}, f\rangle |\> |S_{j}|.
$$
Thus for $f\in {\bf E}_{\omega}(G)$
\begin{equation}
\sum_{j}\langle  \zeta_{j}, f\rangle|S_{j}|-\sum_{v\in V(G)} f(v)\leq \gamma\sum_{v\in V(G)}|f(v)| \leq \frac{\gamma}{1-\gamma}\sum_{j}|\langle  \zeta_{j}, f\rangle |\>|S_{j}|.
\end{equation}
From this one, by adding $ \frac{\gamma}{1-\gamma}\sum_{j}\langle  \zeta_{j}, f\rangle|S_{j}|$ to each side of the inequality
\begin{equation}
-\frac{\gamma}{1-\gamma}\sum_{j}|\langle \zeta_{j}, f\rangle |\>|S_{j}|\leq \sum_{v\in V(G)} f(v)-\sum_{j}\langle  \zeta_{j}, f\rangle|S_{j}|,
\end{equation}
we are getting the next inequality for  every $\>f\in {\bf E}_{\omega}(G)$
\begin{equation}\label{formula1}
\frac{\gamma}{1-\gamma}\sum_{j}\left( \langle  \zeta_{j}, f\rangle -|\langle  \zeta_{j}, f\rangle|\right)|S_{j}|\leq \sum_{v\in V(G)} f(v)-\frac{1-2\gamma}{1-\gamma}\sum_{j}\langle  \zeta_{j}, f\rangle|S_{j}|.
\end{equation}
If a vector $\{f_{j}\}$ belongs to $\mathcal{F}\cap (\mathcal{C}+\mathcal{V})$ then $\{f_{j}\}\in \mathcal{R}(Q)$ and $\{f_{j}\}=\{g_{j}\}+\{h_{j}\}$ where $\{g_{j}\}\in \mathcal{V}, \>\{h_{j}\}\in \mathcal{C}$.
In this situation  for  the following functions in ${\bf E}_{\omega}(G)$:
$$
f=Q^{-1}(\{f_{j}\}),\>g=Q^{-1}(\{g_{j}\}),\>h=Q^{-1}(\{h_{j}\}),
$$
uniquely determined by their sets of samples
$$
f_{j}=\langle\zeta_{j}, f\rangle,\>\>g_{j}=\langle\zeta_{j}, g\rangle,\>\>h_{j}=\langle\zeta_{j}, h\rangle,\>\>\>\>1\leq j\leq |J|,
$$
we have
$$
\Theta(\{f_{j}\})=\Theta\left(\{g_{j}+h_{j}\}\right)=\sum_{v\in V(G)}(g+h)(v) -\frac{1-2\gamma}{1-\gamma}|||\{g_{j}+h_{j}\}|||=
$$ 
$$
\sum_{v\in V(G)}(g+h)(v) -\frac{1-2\gamma}{1-\gamma}\sum_{j}(g_{j}+h_{j})|S_{j}|,
$$
and since $h_{j}\geq 0$, we obtain
$$
\Theta\left(\{g_{j}+h_{j}\}\right)\geq \frac{\gamma}{1-\gamma}\sum_{j}\left( g_{j}  + h_{j}-|g_{j}+h_{j}|\right)|S_{j}|\geq
$$
$$
 \frac{\gamma}{1-\gamma}\sum_{j}\left( g_{j} -|g_{j}|\right)|S_{j}|\geq  -\frac{2}{1-\gamma}\sum_{j} |g_{j}||S_{j}|\geq -\frac{2}{1-\gamma}.
$$
Theorem is proven.
\end{proof}
We are ready to prove our  cubature formula with positive weights.

\begin{thm} If the functionals  $\langle \zeta_{j}, \cdot\rangle, \>j\in J, $ are positive then for a given $0<\gamma<1/2$ there exist weights
$$
\frac{1-2\gamma}{1-\gamma}|S_{j}|\leq \varpi_{j}\leq  \frac{3-2\gamma}{1-\gamma}|S_{j}|,\>\>\>0<\gamma<1/2.
$$
such that the following formula holds for all $f\in {\bf E}_{\omega}(G)$  
$$
\sum_{v\in V(G)} f(v)=\sum_{j}\varpi_{j}\langle  \zeta_{j}, f\rangle,
$$
where  $0\leq \sqrt{ \omega} \leq \gamma/C_{\Xi}\>$, and $\>
C_{\Xi}=\sqrt{    |J|\max_{j\in J}\left(  |S_{j}|\Theta_{ j} \right)  }.
$
\end{thm}
\begin{proof}
Theorems \ref{BN} and \ref{lowerbound} show that the functional $\Theta$ which was originally defined on the subspace $\mathcal{F}=\left(\mathcal{R}(Q),\>||| \cdot |||\right)$ can be extended to a functional $\widetilde{\Theta}$ defined 
on the entire space $\mathcal{E}=\left(\mathbb{R}^{|J|},\>|||\cdot |||\right)$. Moreover, $\widetilde{\Theta}$  will  is positive on $\mathcal{C}$ and bounded by  $\frac{2}{1-\gamma}$. It implies that there exists a vector $\{\sigma_{j}\}_{j\in J}$  such that 
$$
0\leq \sigma_{j}<\frac{2}{1-\gamma},
$$ 
and for which
$$
\widetilde{\Theta}(\{\alpha_{j}\})=\sum_{j}\sigma_{j}\alpha_{j}|S_{j}|,
$$
for every $\{\alpha_{j}\}_{j\in J}\in \mathcal{E}.$
 In particular, for a given $f\in {\bf E}_{\omega}(G)$ one has according to (\ref{functional}) 
$$
\Theta(f)=\sum_{v\in V(G)} f(v)  -\frac{1-2\gamma}{1-\gamma}\sum_{j}\langle  \zeta_{j}, f\rangle |S_{j}|=\sum_{j}\sigma_{j}\langle  \zeta_{j}, f\rangle|S_{j}|,
$$
or
$$
\sum_{v\in V(G)} f(v)=\sum_{j}\varpi_{j}\langle  \zeta_{j}, f\rangle,\>\>\>\>\>f\in {\bf E}_{\omega}(G),
$$
where 
$$
\varpi_{j}=\left(\frac{1-2\gamma}{1-\gamma}+\sigma_{j}\right)|S_{j}|.
$$
Since
$$
\frac{1-2\gamma}{1-\gamma}\leq \left(\frac{1-2\gamma}{1-\gamma}+\sigma_{j}\right)\leq  \frac{3-2\gamma}{1-\gamma},\>\>\>0<\gamma<1/2,
$$
we complete the proof.

\end{proof}

\section{Appendix}

\subsection{Proof of Lemma \ref{Lemma}}\label{Appendix}

  \begin{proof}
For the Fourier coefficients 
$\{c_{k}(f)=\langle f, \varphi_{k}\rangle\}$ one has 
$$
f=\sum_{k=0}c_{k}(f)\varphi_{k}
$$
 and  then if  $\Psi(g)=\left< \psi, g\right>$ we have
 $$
 0=\Psi(f)=c_{0}(f)\Psi(\varphi_{0})+\sum_{k=1}c_{k}(f)\Psi(\varphi_{k}).
 $$
 Using the Parseval equality and Schwartz inequality we obtain
\begin{equation}\label{f-la}
\|f\|^{2}_{H}\left|\Psi(\varphi_{0})\right|^{2}=|c_{0}(f)|^{2}\left|\Psi(\varphi_{0})\right|^{2}+\left|\Psi(\varphi_{0})\right|^{2}\sum_{k=1}|c_{k}(f)|^{2}=
$$
$$
\left|\sum_{k=1}c_{k}(f)\Psi(\varphi_{k})\right|^{2}+\left|\Psi(\varphi_{0})\right|^{2}\sum_{k=1}|c_{k}(f)|^{2}\leq
$$
$$
\sum_{k=1}|c_{k}(f)|^{2} \sum_{k=1}|\Psi(\varphi_{k})|^{2}            +\left|\Psi(\varphi_{0})\right|^{2}\sum_{k=1}|c_{k}(f)|^{2}.
\end{equation}
At the same time we have
$$
\psi=\Psi(\varphi_{0})\varphi_{0}+\sum_{k=1}\Psi(\varphi_{k})\varphi_{k},
$$
and from the Parseval formula
$$
\sum_{k=1}|\Psi(\varphi_{k})|^{2}=\|\psi\|^{2}_{H}-\left|\Psi(\varphi_{0})\right|^{2}.
$$
We plug the right-hand side of this formula into (\ref{f-la}) and obtain the following inequality 
$$
\left|\left<\psi, \varphi_{0}\right>\right|^{2}\|f\|^{2}_{H}\leq\|\psi\|^{2}_{H}\sum_{k=1}|c_{k}(f)|^{2}
\leq  \frac{\|\psi\|^{2}_{H}}{\lambda_{1}^{2}}\sum_{k=1}|\lambda_{k}c_{k}(f)|^{2}= \frac{\|\psi\|^{2}_{H}}{\lambda_{1}^{2}}\|Tf\|^{2}_{H}.
$$
Lemma is proven.
\end{proof}

\subsection{Hilbert frames} Frames in Hilbert spaces were introduced in \cite{DS}.
\begin{defn}
A set of vectors $\{\pi_{j}\}_{j\in J}$  in a Hilbert space $\mathcal{H}$ is called a frame if there exist constants $A, B>0$ such that for all $f\in \mathcal{H}$ 
\begin{equation}
A\|f\|^{2}_{2}\leq \sum_{j\in J}\left|\left<f,\pi_{j}\right>\right|^{2}     \leq B\|f\|_{2}^{2}.
\end{equation}
The largest $A$ and smallest $B$ are called lower and upper frame bounds.  
\end{defn}

The set of scalars $\{\left<f,\pi_{j}\right>\}$ represents a set of measurements of a signal $f$. To synthesize the signal $f$ from this set of measurements one can use  a so-called  (dual) frame $\{\Pi_{j}\}$ and then a reconstruction formula will look like  
\begin{equation}\label{Hfr}
f=\sum_{j\in J}\left<f,\pi_{j}\right>\Pi_{j}.
\end{equation}
Dual frames  are  not unique in general.  Moreover, it is difficult   to find a dual frame.
However, for frames with $A=B=1$ the decomposition and synthesis of functions can be done with the same frame. In other words
 \begin{equation}
f=\sum_{j\in J}\left<f,\pi_{j}\right>\pi_{j}.
\end{equation}
Such frames are known as Parseval (or tight) frames.
For example,
 three vectors in $\mathbf{R}^{2}$ with angles  $2\pi/3$ between them whose lengths are all  $\sqrt{2/3}$ form a Parseval frame.


\bigskip






\begin{thebibliography}{99}

\bibitem{Bab}
C. Babecki, {\em Codes, cubes and graph designs},  J. Fourier Anal. Appl. 27(5), Paper No. 81 (2021).


\bibitem{BabS}
C. Babecki, D. Shiroma, {\em Structure and Complexity of graphical designs for weighted graphs through eigenpolytopes},  arXiv:2209.06349v1 [math.CO] .

\bibitem{BabT}
C. Babecki, R. Thomas, {\em Graphical designs and gale duality}, Mathematical Programming 200, 703-737 (2023).



\bibitem{B}
 H. ~Bauer, {\em Uber die Fortsetzung positiver Linearformen}, Bayer. Akad. Wiss. Math.-Natur. Kl. S.-B., (1957), 177-190.
 
 
 \bibitem{Cav}
 R. Cavoretto, A. De Rossi, W. Erb, {\em GBFPUM-a MATLAB package for partition of unity based signal interpolation and approximation on graphs},  Dolomites Res. Notes Approx. 15 (2022), Special Issue SA2022-Software for Approximation 2022, 25-34.
 
 
 \bibitem{CM}
 A. Cloninger, H. Mhaskar, {\em A low discrepancy sequence on graphs},  J. Fourier Anal. Appl. 27 (2021),  no. 5, Paper No. 76, 27 pp.
 
\bibitem{CKP}
T.~{C}oulhon, G.~{K}erkyacharian, and P.~{P}etrushev.
{\em Heat kernel generated frames in the setting of Dirichlet spaces},
J. Fourier Anal. Appl., 18(5) (2012), 995-1066.

\bibitem {DS}
R. ~Duffin, A. ~Schaeffer, {\em A class of nonharmonic Fourier
series}, Trans. AMS, 72, (1952), 341-366.

\bibitem{Erb}
W. Erb,  {\em Graph signal interpolation with positive definite graph basis functions},  Appl. Comput. Harmon. Anal. 60 (2022), 368-395.
 


\bibitem{F}
S. ~Fortunato, {\em Community detection in graphs}, Phys. Rep., vol. 486,
no. 3-5, (2010),  75-174.



\bibitem{HL}
J. Hogan,  J. Lakey, {\em Spatio-spectral limiting on redundant cubes: a case study},Excursions in harmonic analysis. Vol. 6-in honor of John Benedetto's 80th birthday, 97-115, Appl. Numer. Harmon. Anal., Birkhäuser/Springer, Cham, [2021], ©2021.



\bibitem{LS} 
G. Linderman; S. Steinerberger, {\em  Numerical integration on graphs: where to sample and how to weigh}, Math. Comp. 89 (2020), no. 324, 1933-1952.

\bibitem{N}
 J. ~Namioka, {\em Partially ordered linear topological spaces}, Amer. Math. Soc. Memoir, No. 24, Providence, 1957.
 
  \bibitem{Ortega}
 A. Ortega, P. Frossard, J. Kovacevic, J. Moura, P. Vandergheynst, {\em Graph Signal Processing: Overview, Challenges and Applications}, Proceedings
of the IEEE, pp. 808-828, 2018.


\bibitem{Pes}
I. ~Pesenson, {\em  Sampling in Paley-Wiener spaces on combinatorial graphs}, Trans. Amer. Math. Soc. 360 (2008), no. 10, 5603-5627. 
 
 \bibitem{PPF}
 I. Z. Pesenson, M. Z. Pesenson, and H. Führ, {\em Cubature formulas on combinatorial graphs}, arXiv:1104.0963, 2011.
 
\bibitem{Pes21}I. Z.~Pesenson,  {\em Sampling by averages and average splines on Dirichlet spaces and on combinatorial graphs}. Excursions in harmonic analysis. Vol. 6-in honor of John Benedetto's 80th birthday, 243-268, Appl. Numer. Harmon. Anal., Birkhäuser/Springer, Cham, [2021], ©2021

\bibitem{PesM21}I. Z. ~Pesenson, M. Z. ~Pesenson,  {\em Graph signal sampling and interpolation based on clusters and averages}. J. Fourier Anal. Appl. 27 (2021), no. 3, Paper No. 39, 28 pp.

\bibitem{StT}
 S. Steinerberger, R. R. Thomas, {\em Random walks, Equidistribution and Graphical Designs}, arXiv:2206.05346v1 [math.CO]  

\bibitem{SS}
S. Steinerberger, {\em Generalized designs on graphs: sampling, spectra, symmetries}, J. Graph Theory 93 (2020), no. 2, 253-267. 


\bibitem{Tan}
Y. Tanaka, Y. Eldar, A. Ortega, G. Cheung, {\em Sampling signals on graphs},  IEEE Signal Processing Magazine,  Volume: 37, Issue: 6, November 2020.


\bibitem{TB}
N. Tremblay,  P. Borgnat, {\em Subgraph-Based Filterbanks for Graph Signals},  EEE Trans. Signal Process., Vol. 64, No. 15, August 1, 2016. 

\bibitem{WLG}X. Wang, P. Liu, Y. Gu,
{\em Local-Set-Based Graph Signal Reconstruction},
IEEE Transactions on Signal Processing (2015).

\bibitem{WCG}
X. Wang, J. Chen, Y. Gu, {\em Local measurement and reconstruction for noisy bandlimited graph
signals},  Signal  Process. 129, 119-129 (2016).
\end{thebibliography}
\end{document}